\documentclass[11pt]{article}

\usepackage{amsfonts,amsmath,amssymb}
\usepackage{amsmath,amsthm}
\usepackage{latexsym}
\usepackage{color}
\usepackage{tikz}

\newtheorem{thm}{Theorem}

\newtheorem{lem}[thm]{Lemma}

\newtheorem{cor}[thm]{Corollary}
\newtheorem{prop}[thm]{Proposition}

\newtheorem{prob}{Problem}

\newcommand{\DP}{\times}

\newcommand{\LSs}[2]{^{#1}{#2}}

\begin{document}

\title{On well-covered direct products}

\author{
$^{a}$Kirsti Kuenzel
\and
$^{b}$Douglas F. Rall\\
}

\date{\today}

\maketitle

\begin{center}
$^a$ Department of Mathematics, Western New England University, Springfield, MA\\
$^b$ Department of Mathematics, Furman University, Greenville, SC\\

\end{center}

\begin{abstract}
A graph $G$ is well-covered if all maximal independent sets of $G$ have the same cardinality.
In 1992 Topp and Volkmann investigated the structure of well-covered graphs that have nontrivial factorizations with respect
to some of the standard graph products.  In particular, they showed that both factors of a well-covered direct product are
also well-covered and proved that the direct product of two complete graphs (respectively, two cycles) is well-covered precisely when
they have the same order (respectively, both have order 3 or 4).  Furthermore, they proved that the direct product of two well-covered
graphs with independence number one-half their order is well-covered.  We initiate a characterization  of
nontrivial, connected well-covered graphs $G$ and $H$, whose
independence numbers are strictly less than one-half their orders, such that their direct product $G \times H$ is well-covered.  In particular, we show that
in this case both $G$ and $H$ have girth 3 and we present several infinite families of such well-covered direct products.
Moreover, we show that if $G$ is a factor of any well-covered direct product, then $G$ is a complete graph unless it is possible to
create an isolated vertex by removing the closed neighborhood of some independent set of vertices in $G$.
\end{abstract}

\noindent {\bf Key words:} well-covered graph; direct product of graphs; isolatable vertex

\medskip\noindent
{\bf AMS Subj.\ Class:} 05C69; 05C76

%%%%%%%%%%%%%%%%%%%%%%
\section{Introduction}
%%%%%%%%%%%%%%%%%%%%%%

Plummer~\cite{P-1970} defined a graph to be well-covered if every maximal independent set is actually a maximum independent set.
The attempts to better understand the class of well-covered graphs have, for the most part, proceeded as follows.  Find
a nice characterization of those well-covered graphs that, in addition, belong to some natural subclass of graphs.  For instance,
Campbell, Ellingham and Royle~\cite{CER-1993} characterized the class of cubic well-covered graphs. Finbow, Hartnell and Nowakowski~\cite{FHN-1993}
characterized well-covered graphs that have no cycles of order less than 5; the same group of authors~\cite{FHN-1994} dealt with well-covered
graphs with no cycles of length 4 or 5.  In a series of papers~\cite{FHNP-1,FHNP-2,FHNP-3,FHNP-4} Finbow, Hartnell, Nowakowski and Plummer gave
a complete characterization of the class of maximal planar, well-covered graphs.

Topp and Volkmann~\cite{TV-1992} first studied well-covered graphs in the context of graph products, including the Cartesian, conjunction (now commonly
known as direct), and lexicographic products.  From their study open questions remained for Cartesian and direct products.  Several authors
contributed to the current understanding of well-covered Cartesian products.  See~\cite{F-2009, HR-2013,HRW-2018+}.  As far as well-covered
direct products are concerned, Topp and Volkmann focused mainly on graphs whose independence number is one-half the order.  These graphs are called very well-covered.
However, much remains unknown about direct products that are well-covered but not very well-covered.  In this paper we initiate the characterization of
this class of graphs.

The remainder of the paper is structured as follows. In the next section we provide the important definitions and recall preliminary results that will be
used in the remainder of the paper.  Section~\ref{sec:onefactorcomplete} is devoted to direct products in which one of the factors is a complete graph.  In
Section~\ref{sec:noisolatables} we focus on direct products in which one of the factors has no isolatable vertices.  In the main result of this section we
prove that if $G \times H$ is well-covered and $G$ has no isolatable vertices, then $G$ is in fact a complete graph.  In addition, for each positive integer $n \ge 3$
we provide two infinite families of graphs such that the direct product of $K_n$ and any graph from these families is well-covered.  In Section~\ref{sec:general}
we prove that if $G \times H$ is well-covered but not very well-covered, then every edge of $G$ (and of $H$) is incident with a triangle.
In particular, in this case both factors have girth $3$.

%%%%%%%%%%%%%%%%%%%%%%%%%%%%%%%%%%%%%%%%%%%%%
\section{Definitions and preliminary results} \label{sec:prelim}
%%%%%%%%%%%%%%%%%%%%%%%%%%%%%%%%%%%%%%%%%%%%%

In general we follow the notation of~\cite{W-1996}. In particular, we denote the order of a finite graph $G$ by $n(G)$ and for a positive integer
$k$ the set of positive integers no larger than $k$ will be denoted by $[k]$.  If $A \subseteq V(G)$, then $G[A]$ is the subgraph of $G$
induced by $A$.  The set of isolated vertices of $G$ will be denoted $G_0$ and $G^+$ will represent the induced subgraph $G-G_0$.  A subset
$D\subseteq V(G)$ \emph{dominates} a subset $S \subseteq V(G)$ if $S \subseteq N[D]$.  If $D$ dominates~$V(G)$, then we will also say that
$D$ \emph{dominates the graph} $G$ and that $D$ is a \emph{dominating set} of $G$.  A set $I \subseteq V(G)$ is an \emph{independent dominating} set
if $I$ is simultaneously independent and dominating.  This is equivalent to $I$ being a maximal independent set with respect to set inclusion.
The \emph{independence number} of $G$ is the cardinality, $\alpha(G)$, of a largest independent set in $G$; we denote the smallest cardinality
of a maximal independent set in $G$ by $i(G)$.  The graph $G$ is \emph{well-covered} if all maximal independent sets of $G$ have
the same cardinality.  Equivalently, $G$ is well-covered if $i(G)=\alpha(G)$. The \emph{independence ratio} of a graph $G$
is defined by $\frac{\alpha(G)}{n(G)}$.

In a well-covered graph $G$ every vertex can (in a greedy fashion) be enlarged to a maximal independent set, which then has order $\alpha(G)$.
Note that a graph is well-covered if and only if each of its  components is well-covered.  A vertex of degree 1 is called
a \emph{leaf} and its only neighbor is called a \emph{support vertex}.  If $G$ is a well-covered graph with a support vertex $x$ and
$M$ is any maximal independent set in $G$ that contains $x$, then replacing $x$ in $M$ by its set $L$ of adjacent leaves is also independent.  It follows that
$|L|=1$.  A vertex $x$ of $G$ is \emph{isolatable} if there exists an independent set $I$ in $G$ such that $x$ has degree 0 (that is,
$x$ is \emph{isolated}) in $G-N[I]$.  Note that a leaf in a component of order at least 3 is isolatable.

The \emph{direct product}, $G\DP H$, of graphs $G$ and $H$ is defined as follows:
\begin{itemize}
\item $V(G \DP H)=V(G) \DP V(H)$;
\item $E(G \DP H)= \{ (g_1,h_1)(g_2,h_2) \, |\, g_1g_2 \in E(G) \,\,\text{and}\,\,h_1h_2 \in E(H) \}$
\end{itemize}
The direct product is both commutative, associative and distributes over disjoint unions of graphs.   For a vertex $g$ of $G$, the \emph{$H$-layer over $g$}
of $G\DP H$ is the set $\{ \, (g,h) \mid h\in V(H) \,\}$, and it is denoted by $\LSs g H$.  Similarly, for $h \in V(H)$, the \emph{$G$-layer over $h$},
$G^h$, is the set $\{ \, (g,h) \mid g\in V(G) \,\}$.  Note that each $G$-layer and each $H$-layer is an independent set in $G\DP H$.  The \emph{projection to $G$}
is the map $p_G: V(G\DP H) \to V(G)$ defined by $p_G(g,h)=g$.  Similarly, the \emph{projection to $H$} is the map $p_H: V(G\DP H) \to V(H)$ defined by $p_H(g,h)=h$.

In the remainder of this section we present some results that will prove useful in establishing our main results.  The first lemma is due to Topp and Volkmann~\cite{TV-1992}.
We provide a short proof since the ideas therein are so common when studying well-covered direct products.

\begin{lem}[\cite{TV-1992}] \label{lem:inverseimage}
Let $H$ be a graph with no isolated vertices.  If $I$ is a maximal independent set of any graph $G$, then $I \DP V(H)$ is a maximal independent set of $G \DP H$.
\end{lem}
\begin{proof} For any $g \in I$, the $H$-layer over $g$ is independent.  Since $I$ is independent in $G$, it follows that for distinct vertices $a$ and $b$ in $I$
no vertex of $\LSs {a}{H}$ is adjacent to any vertex of $\LSs{b}{H}$, and thus $I \DP V(H)$ is independent.  Let $(u,v) \in V(G\DP H)-(I \DP V(H))$.  Since
$I$ is a maximal independent set of $G$ and $u\not\in I$, we infer there exists $x\in I$ such that $x$ and $u$ are adjacent.  For any neighbor $y$ of $v$
(such a vertex $y$ exists since $H$ has no isolated vertices), it
follows that $(x,y)$ belongs to $I \DP V(H)$ and is adjacent to $(u,v)$.  We conclude that $I \DP V(H)$ is a maximal independent set in $G \DP H$.  \end{proof}

\medskip
As an immediate consequence of Lemma~\ref{lem:inverseimage} we get a lower bound for $\alpha(G \DP H)$, which is well-known
(see \cite{JS-1994, NR-1996}), and an upper bound for $i(G \DP H)$.
\begin{cor} \label{cor:triviallower}
If both $G$ and $H$ have no isolated vertices, then
\begin{itemize}
\item $\alpha(G \DP H) \ge \max \{\alpha(G) n(H), \alpha(H) n(G)\}$;
\item $i(G \DP H) \le \min \{i(G) n(H), i(H) n(G)\}$.
\end{itemize}
\end{cor}

The following lemma follows directly from the definition of well-covered.  It has been very useful (especially as a necessary condition
to show that a graph is not well-covered) in several of the papers characterizing well-covered graphs having some additional property (for example, a girth restriction).

\begin{lem} [\cite{FHN-1993}] \label{lem:basic}
If $G$ is a well-covered graph and $I$ is an independent set of $G$, then $G-N[I]$ is well-covered.
\end{lem}

The following lemma holds for any graph.
\begin{lem} \label{lem:cliqueleftover}
If $G$ is any  graph and $J$ is an independent set of vertices in $G$ such that $|J|=\alpha(G)-1$, then either $J$
is a maximal independent set or $G-N[J]$ is a complete graph.
\end{lem}
\begin{proof}
Suppose that $J$ is not a maximal independent set in $G$.  This implies that $G - N[J]$ is nonempty. If $G-N[J]$ contains two nonadjacent
vertices $u$ and $v$, then $J \cup \{u,v\}$ is independent and has cardinality $\alpha(G)+1$, which is a contradiction.
\end{proof}

Our results will always involve graphs with no isolated vertices.  However, there are a number of situations in which isolated vertices
arise when the closed neighborhood of an independent set is removed from a graph.  Thus we have the following generalization of a theorem
first proved by Topp and Volkmann~\cite{TV-1992}.
\begin{thm}\label{thm:wcdirect}
If $G$ and $H$ are graphs and $G\times H$ is well-covered, then
\begin{enumerate}
\item[(a)] $G$ and $H$ are well-covered, and
\item[(b)] $\alpha(G^+)n(H^+) = \alpha(H^+)n(G^+)$.
\end{enumerate}
\end{thm}
\begin{proof}
Assume that $G \DP H$ is well covered. Since the direct product distributes over disjoint unions, $G \DP H$ is the disjoint union of $G^+ \DP H^+$ and a graph
$K$, which is a set of $n(G)\cdot |H_0| +|G_0|\cdot n(H^+)$ isolated vertices.  The subgraph $G^+ \DP H^+$ of $G \DP H$ is well-covered since it is
the disjoint union of some components (possibly just 1) of the well-covered graph $G \DP H$.   If $I$ is any maximal independent
set of $G^+$, then by Lemma~\ref{lem:inverseimage} we get $(I \times V(H^+)) \cup K$  is a maximal independent set of $G \times H$.  Using the assumption that
$G \times H$ is well-covered, this implies that all maximal independent sets of $G^+$ have the same cardinality and so $G^+$ and $G$ are well-covered.
Similarly, $H^+$ and $H$ are well-covered.  Moreover, both $G^+$ and $H^+$ have no isolated vertices.  We infer by Corollary~\ref{cor:triviallower} that
$\alpha(G^+)n(H^+)=\alpha(H^+)n(G^+)$.
\end{proof}

When $G \DP H$ is well-covered and neither $G$ nor $H$ has isolated vertices, Theorem~\ref{thm:wcdirect} implies that $G$ and $H$ have the same \emph{independence ratio}.
That is, $\frac{\alpha(G)}{n(G)}=\frac{\alpha(H)}{n(H)}$.
On the other hand, if $G$ or $H$ has isolated vertices and $G \DP H$ is well-covered, then these ratios may not be equal.
For a small example let $G=K_2 \cup K_1$ and $H=K_2$.

The following result was proved by Berge~\cite{b-1981}.

\begin{thm}[\cite{b-1981}]
If $G$ is a well-covered graph with no isolated vertices, then $|S| \le |N(S)|$ for any independent set $S$ of $G$.
\end{thm}

This result immediately implies that $\alpha(G) \le \frac{1}{2}n(G)$, for any well-covered $G$ with no isolated vertices.
A well-covered graph with no isolated vertices that achieves this upper bound is called \emph{very well-covered}.  Since both partite sets
of a bipartite graph are maximal independent sets, it is clear that a well-covered bipartite graph with no isolated vertices is very well-covered.
Favaron~\cite{f-1982} characterized the very well-covered graphs in terms of the existence of a perfect matching that possesses a special property.
Let $G$ be a graph with a perfect matching $M$.  For each vertex $u$ of $G$, we let $M(u)$ denote the vertex adjacent to $u$ in $M$. Favaron~\cite{f-1982}
said $M$ has \emph{Property (P)} if for every vertex $x$ of $G$ the following holds.
\begin{itemize}
\item If $y\in N_G(x)$ and $y\neq M(x)$, then $y \notin N_G(M(x))$ and $y \in N_G(z)$, for every $z \in N_G(M(x))$.
\end{itemize}
The following theorem of Favaron gives the aforementioned characterization.
\begin{thm} [\cite{f-1982}] \label{thm:vwc-characterization}
The following are equivalent for any simple graph $G$.
\begin{enumerate}
\item[(i)] The graph $G$ is very well-covered.
\item[(ii)] There is a perfect matching in $G$  that satisfies Property (P).
\item[(iii)] There exists at least one perfect matching in $G$, and every perfect matching in $G$ satisfies Property (P).
\end{enumerate}
\end{thm}

We will need the following theorem of Topp and Volkmann concerning very well-covered graphs.

\begin{thm}[\cite{TV-1992}]\label{thm:verywcdirect} Let $G$ and $H$ be graphs without isolated vertices. If at least one of $G$ and $H$ is very
well-covered, then the following statements are equivalent:
\begin{enumerate}
\item[(a)] $G\times H$ is well-covered,
\item[(b)] $G\times H$ is very well-covered,
\item[(c)] both $G$ and $H$ are very well-covered.
\end{enumerate}
\end{thm}

Because of Theorem~\ref{thm:verywcdirect} the general problem of characterizing well-covered direct products that are not very well-covered
is reduced to characterizing those pairs of well-covered graphs $G$ and $H$, neither of which is very well-covered, but whose direct product $G \DP H$
is well-covered.

%%%%%%%%%%%%%%%%%%%%%%%%%%%%%%%%%%%%%%%%%%%%%%%%%%%
\section{Products of the form $G\times K_n$} \label{sec:onefactorcomplete}
%%%%%%%%%%%%%%%%%%%%%%%%%%%%%%%%%%%%%%%%%%%%%%%%%%%

Suppose that $I$ is a maximal independent set in $G \DP H$ and that $g$ is a vertex of $G$ such that $I \cap {\LSs{g}{H}} \not=\emptyset$ but that $I \cap {\LSs{g}{H}} \not= {\LSs{g}{H}}$.  Let $(g,h)\in {\LSs{g}{H}- (I \cap {\LSs{g}{H}})}$.  Since $I$ is a dominating set of $G \DP H$ and $\LSs{g}{H}$ is independent, it follows that there exists $g' \in N_G(g)$ and $h'\in N_H(h)$ such that $(g',h') \in I$.  Furthermore, such a vertex $h'$ does not belong to $N_H(p_H(I \cap {\LSs{g}{H}}))$.  However, it is possible that $h' \in p_H(I \cap {\LSs{g}{H}})$.

Consider now the special case $G \DP K_n$ for $n \ge 2$.  Let $V(K_n)=[n]$.

\begin{lem} \label{lem:sizeoflayers}
Let $n \ge 2$ and let $G$ be any graph.  If $I$ is any maximal independent set of $G \DP K_n$, then $\left| I \cap {\LSs{g}{K_n}} \right | \in \{0,1,n\}$, for any $g \in V(G)$.
\end{lem}
\begin{proof} If $n=2$, then the conclusion is obvious.  Assume $n\ge 3$ and suppose for the sake of contradiction that $\left| I \cap {\LSs{g}{K_n}} \right |=m$ for some $2 \le m <n$.  Assume without loss of generality that $\{(g,1),(g,2)\} \subseteq I$.  Let $i \in [n]$ such that $(g,i) \not\in I$. As above, there exists $g' \in N_G(g)$ and $j \in N_{K_n}(i)$ such that $(g',j)\in I$.  Since $n \ge 3$, we infer that $j \not=1$ or $j \not =2$.  This implies that $(g',j) \in N(\{(g,1),(g,2)\})$, which contradicts the independence of $I$.  Therefore,
$\left| I \cap {\LSs{g}{K_n}} \right | \in \{0,1,n\}$.   \end{proof}

For an arbitrary positive integer $n \ge 2$ and a maximal independent set $I$ of $G \DP K_n$, we can use Lemma~\ref{lem:sizeoflayers} to define a weak partition of $V(G)$.  In particular, $V_0,V_1,\ldots,V_n,V_{[n]}$ defined by
\begin{enumerate}
\item[{\bf(a)}] $V_0=\{ g \in V(G)\,|\, I \cap {\LSs{g}{K_n}} = \emptyset\}$;
\item[{\bf(b)}] $V_k=\{ g \in V(G)\,|\, I \cap {\LSs{g}{K_n}} =\{(g,k)\} \}$ for $k \in [n]$;
\item[{\bf(c)}] $V_{[n]}=\{ g \in V(G)\,|\, I \cap {\LSs{g}{K_n}}={\LSs{g}{K_n}} \}$.
\end{enumerate}
is a weak partition.  Furthermore, the following four conditions hold.
\begin{enumerate}
\item For $k \in [n]$, if $u \in V_k$ and $v \in V(G)-(V_0 \cup V_k)$, then $uv \not\in E(G)$.
\item For $k \in [n]$, if $V_k$ is not empty, then no vertex of $V_k$ is isolated in $G[V_k]$.
\item The set $V_{[n]}$ is independent in $G$.
\item For each $g \in V_0$, either $N_G(g) \cap V_{[n]} \not=\emptyset$ or $g$ has a neighbor in at least two of the sets $V_1,\ldots,V_n$.
\end{enumerate}

If we have a weak partition of $V(G)$ that satisfies these four conditions, then it is clear how to construct a maximal independent set of $G \DP K_n$.
Thus we have a way to define $i(G \DP K_n)$ and $\alpha(G \DP K_n)$ in terms of such partitions.
\[i(G \DP K_n)= \min\{n\cdot |V_{[n]}|+ \sum_{k=1}^n|V_k|\}\hskip 1cm \alpha(G \DP K_n)= \max\{n\cdot |V_{[n]}|+ \sum_{k=1}^n|V_k|\},\]
where the minimum and maximum values are computed over all weak partitions \\
$V_0,V_1,\ldots,V_n,V_{[n]}$ that satisfy conditions $1{-}4$ above.

The next lemma gives a necessary condition on a graph $G$ for the direct product of $G$ and a complete graph to be well-covered.
\begin{lem} \label{lem:necessary3}
Let $n$ be a positive integer, $n \ge 2$.  If $G \DP K_n$ is well-covered, then for every $x\in V(G)$ such that $\deg(x) \ge n$ the
graph $G - N[x]$ has at least one isolated vertex.
\end{lem}
\begin{proof}
We prove the contrapositive.  Suppose $x$ is a vertex in $G$ of degree at least $n$ such that $G - N[x]$ has minimum degree at least 1.  We define
a weak partition $V_0,V_1,\ldots,V_n,V_{[n]}$ as follows.  Let $V_0=N(x)$, let $V_1=V(G)-N[x]$, let $V_i=\emptyset$ for $2 \le i \le n$, and let $V_{[n]}=\{x\}$.
Since $G - N[x]$ has minimum degree at least 1, it is easy to check that this weak partition satisfies conditions $1{-}4$ above. Furthermore,
$\deg(x) \ge n$, and hence,
\[i(G \DP K_n) \le n +|V(G)-N[x]| \le n(G)-1 <n(G) \le \alpha(G \DP K_n)\,.\]
This shows that $G \DP K_n$ is not well-covered.
\end{proof}

Using Lemma~\ref{lem:necessary3} we now use the context of direct products to prove a general result about bipartite well-covered graphs.

\begin{cor} \label{cor:bipartite}
If $B$ is a bipartite, well-covered graph with minimum degree at least $2$, then $B$ has isolatable vertices.  In fact, for any vertex $x$ of $B$,
the induced subgraph $B-N[x]$ has an isolated vertex.
\end{cor}
\begin{proof}
Suppose $B$ is bipartite, well-covered and $\delta(B) \ge 2$.  The graph $B$ is very well-covered since it is bipartite and well-covered.
By Theorem~\ref{thm:verywcdirect}, $B \DP K_2$ is very well-covered.  For any $x \in V(B)$, it follows from Lemma~\ref{lem:necessary3} that
$B-N[x]$ has at least one isolated vertex.  \end{proof}

%%%%%%%%%%%%%%%%%%%%%%%%%%%%%%%%%%%%%%%%%%%%%
\section{Factors with no isolatable vertices} \label{sec:noisolatables}
%%%%%%%%%%%%%%%%%%%%%%%%%%%%%%%%%%%%%%%%%%%%%

As mentioned in \cite{FHN-1993}, when classifying well-covered graphs one particularly useful property of a graph is whether or not
it contains isolatable vertices. We first consider direct products where at least one of the factor graphs does not contain isolatable
vertices.

\begin{lem} \label{lem:structure-no-isolatables}
Let $H$ be a nontrivial, connected graph and let $G$ be a graph with no isolatable vertices such that $G \times H$ is well-covered.
If $k\in [\alpha(G)]$ and $A$ is any independent set of $G$ such that $|A|=k$, then $|N[A]|=k\frac{n(G)}{\alpha(G)}$.
\end{lem}

\begin{proof}
Assume $G$ and $H$ are as in the hypothesis of the theorem, let $1 \le k \le \alpha(G)$ and let $A$ be an independent set of $G$ such
that $|A|=k$.  For $k=\alpha(G)$ the conclusion holds since any independent set of $G$ that has cardinality $\alpha(G)$ dominates $G$.
If $G$ is a clique, then $k=1$ and again the conclusion follows.  Hence, we may assume that $G$ is not a complete graph.
Since both $G$ and $H$ have minimum degree at least 1, it follows from  Theorem~\ref{thm:wcdirect}
that $\alpha(G)n(H) = \alpha(H)n(G)$ and also that both $G$ and $H$ are well-covered.
Since $\delta(H) \ge 1$, $G\times H - N[A \times V(H)]=G' \times H$, where $G'=G-N[A]$. Since $A$ is independent, Lemma~\ref{lem:basic}
implies that $G'$ is well-covered, and since $G$ has no isolatable vertex we infer that $\delta(G') \ge 1$.  Furthermore, $A \times V(H)$
is independent in $G \DP H$, and thus by Lemma~\ref{lem:basic} we conclude that $G' \DP H$ is well-covered. Applying
Theorem~\ref{thm:wcdirect} again gives
\[\frac{\alpha(G)}{n(G)} = \frac{\alpha(H)}{n(H)} =  \frac{\alpha(G')}{n(G')}.\]
Note that since $G$ is well-covered, there exists a maximum independent set of $G$ that contains $A$ and this implies that $\alpha(G')=\alpha(G)-k$.
Thus,
\[\frac{\alpha(G)}{n(G)} = \frac{\alpha(G')}{n(G')}=\frac{\alpha(G)-k}{n(G)-|N[A]|}\,,\]
 which implies $|N[A]|=k\frac{n(G)}{\alpha(G)}$.
\end{proof}

A special case of Lemma~\ref{lem:structure-no-isolatables} is when $k=1$, which yields the following corollary.
\begin{cor} \label{cor:regular}
Let $G$ and $H$ be nontrivial, connected graphs such that $G$ has no isolatable vertex.
If $G \times H$ is well-covered, then $G$ is a regular graph of degree $\frac{n(G)}{\alpha(G)}-1$.
\end{cor}

Using Corollary~\ref{cor:regular} we can now easily establish the following result.
\begin{cor} \label{cor:withk3}
Let $G$ be a nontrivial, connected graph.  If $G \times K_3$ is well-covered, then $G=K_3$ or $G$ has at least one isolatable vertex.
\end{cor}
\begin{proof}
Assume that $G \times K_3$ is well-covered and that $G$ has no isolatable vertex.
Since $G \times K_3$ is well-covered, Theorem~\ref{thm:wcdirect} implies $G$ is well-covered and that the independence ratio,
$\frac{\alpha(G)}{n(G)}$, is $1/3$.  By Corollary~\ref{cor:regular} we see that $G$ is $2$-regular.  However, the only
connected $2$-regular graphs that are well-covered are $K_3, C_4, C_5$ and $C_7$.  Of these, only $K_3$ has independence ratio $1/3$.
\end{proof}

We now have the results necessary to give a partial characterization of well-covered direct products in which at least one of the factors
has no isolatable vertices.

\begin{thm} \label{thm:noisolatable}
Let $G$ and $H$ be nontrivial, connected graphs such that the direct product $G \times H$ is well-covered.  If $G$
has no isolatable vertices, then $G$ is a complete graph.
\end{thm}
\begin{proof}
Let $G$ and $H$ be connected graphs with minimum degree at least 1 such that $G \DP H$ is well-covered and
assume that $G$ has no isolatable vertices.  Suppose that $G$ is not a complete graph.  That is, suppose that $s=\alpha(G)-1 \ge 1$.  Let $x$ be any vertex of $G$.
Since $G$ is well-covered (by Theorem~\ref{thm:wcdirect}), we can find a maximal (in fact a maximum) independent set of $G$ that contains $x$.  Let $M$ be
such a maximum independent set.  By Lemma~\ref{lem:structure-no-isolatables}, $|N[M-\{x\}]|=s\frac{n(G)}{\alpha(G)}=s\frac{n(G)}{s+1}$,
and by Corollary~\ref{cor:regular},
\[|V(G)-N[M-\{x\}]|=n(G)-s\frac{n(G)}{s+1}=\frac{n(G)}{s+1}=\frac{n(G)}{\alpha(G)}=|N[x]|\,.\]
In addition, by Lemma~\ref{lem:cliqueleftover}, $G-N[M-\{x\}]$ is a clique, it contains $x$, and it has order $|N[x]|$.  By Corollary~\ref{cor:regular} $G$ is regular,
and  this implies that $N[x]$ is a component of  $G$.  Since $G$ is connected, we conclude that $G$ is complete, which is a contradiction.
\end{proof}

Using Theorem~\ref{thm:noisolatable} we can completely characterize direct products of connected graphs in which neither factor has an isolatable vertex.

\begin{cor} \label{cor:bothcomplete}
Let $G$ and $H$ be nontrivial, connected graphs, neither of which has an isolatable vertex.  If $G \DP H$ is well-covered, then $G=H=K_{n(G)}$.
\end{cor}
\begin{proof}
Since neither $G$ nor $H$ has an isolatable vertex, it follows from Theorem~\ref{thm:noisolatable} that both $G$ and $H$ are complete graphs.
By Theorem~\ref{thm:wcdirect}, the independence ratios of $G$ and $H$ are equal.  Consequently, $G$ and $H$ have the same order.
\end{proof}
Thus, classifying all well-covered direct products when \emph{exactly} one of the factor graphs does not contain isolatable vertices reduces to the
study of well-covered direct products of the form $K_n \times G$.  Using Lemma~\ref{lem:sizeoflayers} and the partition approach in Section~\ref{sec:onefactorcomplete},
it is easy to show that for any integer $r\ge 2$, the direct product $K_3 \times K_{r,r,r}$ is well-covered.   (This generalizes to $K_n \times K_{r,\ldots,r}$ being
well-covered, where the second factor is a complete $n$-partite graph.)
To see that finding a characterization of those $G$ such that $G \times K_n$ is well-covered is a nontrivial problem, consider the following infinite
class of graphs. Let $k$ and $n$ be  positive integers.  Form a graph $H(k,n)$ of order $k(n+1)$ by starting with the
disjoint union of $K_{kn}$ and an independent  set $\{z_1,\ldots,z_k\}$.  Partition the vertices of $K_{kn}$ into subsets $A_1,\ldots,A_k$ each
of cardinality $n$.  Finally, add edges to make the open neighborhood of $z_i$ in $H(k,n)$ be $A_i$, for each $i \in [k]$.
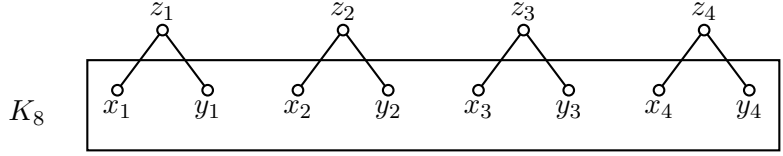
\begin{figure}[ht!]
\begin{center}
\begin{tikzpicture}[scale=0.4,style=thick]
\def\vr{5pt} % \vr = vertex radius;  Set \vr = 2/scale for uniform sizing of vertices
%%%
%%% Vertices Spine
\path (0,0) coordinate (x1); \path (3,0) coordinate (y1); \path (6,0) coordinate (x2); \path (9,0) coordinate (y2);
\path (12,0) coordinate (x3); \path (15,0) coordinate (y3); \path (18,0) coordinate (x4); \path (21,0) coordinate (y4);
\path (1.5,2) coordinate (z1); \path (7.5,2) coordinate (z2); \path (13.5,2) coordinate (z3); \path (19.5,2) coordinate (z4);
%%% Vertices Leaves

%%% Draw edges
\draw (z1) -- (x1); \draw (z1) -- (y1); \draw (z2) -- (x2); \draw (z2) -- (y2);
\draw (z3) -- (x3); \draw (z3) -- (y3);  \draw (z4) -- (x4); \draw (z4) -- (y4);
\draw (-1,-2) rectangle (22,1);
%
%%% Draw vertices
\foreach \i in {1,...,4}
{  \draw (x\i)  [fill=white] circle (\vr); \draw (y\i)  [fill=white] circle (\vr); \draw (z\i)  [fill=white] circle (\vr);}

%%%  Label vertices on Spine
\draw[anchor=south] (z1) node{$z_1$}; \draw[anchor=south] (z2) node{$z_2$}; \draw[anchor=south] (z3) node{$z_3$}; \draw[anchor=south] (z4) node{$z_4$};
\draw[anchor=north] (x1) node{$x_1$}; \draw[anchor=north] (y1) node{$y_1$}; \draw[anchor=north] (x2) node{$x_2$}; \draw[anchor=north] (y2) node{$y_2$};
\draw[anchor=north] (x3) node{$x_3$}; \draw[anchor=north] (y3) node{$y_3$}; \draw[anchor=north] (x4) node{$x_4$}; \draw[anchor=north] (y4) node{$y_4$};

\draw (-3,-.75) node{$K_8$};

\end{tikzpicture}
\end{center}
\caption{The graph $H(4,2)$.} \label{fig:H(4,2)}
\end{figure}

The example in Figure~\ref{fig:H(4,2)} is $H(4,2)$.  For the special case when $n=1$, the resulting graphs $H(k,1)$ are coronas.
If $k=1$, then $H(1,n)=K_{n+1}$.  All of these graphs are split graphs, and because of their structure it is easy to show that $H(k,n)$ is well-covered
with $\alpha(H(k,n))=k$.

\begin{prop} \label{prop:example}
For each pair of positive integers $n$ and $k$,  the graph $H(k,n) \times K_{n+1}$ is well-covered.
\end{prop}
\begin{proof}
We begin by proving the proposition for the case $k=4$ and $n=2$.    For ease of notation
let $H=H(4,2)$ and let $X=\{x_1,x_2,x_3,x_4\}$, $Y=\{y_1,y_2,y_3,y_4\}$ and $Z=\{z_1,z_2,z_3,z_4\}$.
The subgraph induced by $X \cup Y$ is a complete graph, and  it is clear that $\alpha(H)=4$.
Let $V_0,V_1,V_2,V_3,V_{[3]}$ be a partition of $V(H)$ that satisfies conditions $(1)-(4)$ in Section~3.
We claim that for any such partition, $3\cdot |V_{[3]}|+ \sum_{i=1}^3|V_i|=12$;  that is, we claim that
\[i(H\DP K_3)=\alpha(H \DP K_3)=12\,,\]
and thus that $H \DP K_{2+1}$ is well-covered.

Suppose first that $V_0 \cap Z \neq \emptyset$.  Without loss of generality we assume that $z_1\in V_0$.  By condition (4), either
$N(z_1) \cap V_{[3]}\neq \emptyset$ or for some pair of distinct indices $i,j$, $x_1 \in V_i$ and $y_1 \in V_j$.  This latter condition
is not possible by condition (1).  Suppose, then, without loss of generality, that $x_1 \in V_{[3]}$.  This implies that
$(X \cup Y)-\{x_1\} \subseteq V_0$ by conditions (1) and (3).  Using conditions (2) and (4) we conclude that
$V_{[3]}=Z$, and the claim is established in this case.

Suppose next that $V_0 \cap Z=\emptyset$.  In this case suppose first that $z_1 \in V_1$.  By condition (2), $V_1 \cap \{x_1,y_1\} \neq \emptyset$, and
then by condition (1) $\{x_2,y_2,x_3,y_3,x_4,y_4\} \subseteq V_0 \cup V_1$.  This latter conclusion implies (by condition (4)) that $\{x_1,y_1\} \subseteq V_1$.
Let $2 \le j \le 4$.  If $\{x_j,y_j\} \cap V_1 \neq \emptyset$, then $\{x_j,y_j,z_j\} \subseteq V_1$.  On the other hand, if $\{x_j,y_j\} \cap V_1 = \emptyset$,
then $\{x_j,y_j\} \subseteq V_0$ and $z_j \in V_{[3]}$.  We infer that $3\cdot |V_{[3]}|+ \sum_{i=1}^3|V_i|=12$.  A symmetric argument shows $3\cdot |V_{[3]}|+ \sum_{i=1}^3|V_i|=12$
holds if $z_1 \in V_2 \cup V_3$.

Finally, suppose that $z_1 \in V_{[3]}$.  By conditions (1) and (4), we see that $\{x_1,y_1\} \subseteq V_0$.
If $\{z_2,z_3,z_4\} \subseteq V_{[3]}$, then $X \cup Y=V_0$, and we get $3\cdot |V_{[3]}|+ \sum_{i=1}^3|V_i|=12$.  On the other hand, if for example $z_2 \in V_1$,
then $\{x_2,y_2\} \subseteq V_1$ and for each $j\in \{3,4\}$ we have either $\{x_j,y_j,z_j\} \subseteq V_1$ or ($z_j \in V_{[3]}$ and $\{x_j,y_j\}\subseteq V_0$).
Again in this case we conclude that $3\cdot |V_{[3]}|+ \sum_{i=1}^3|V_i|=12$.

Having examined all possible cases we conclude that $H(4,2) \times K_{2+1}$ is well-covered.  The proof for arbitrary $n$ and $k$ is similar but has more cases and
is omitted.  \end{proof}

\begin{prob}
Let $n$ be a positive integer.  Find a characterization of the class, $\cal{C}$, of all connected graphs $G$
such that $G$ has an isolatable vertex and $G \DP K_n$ is well-covered.
\end{prob}

%%%%%%%%%%%%%%%%%%%%%%%%%%%%%%%%%
\section{General direct products} \label{sec:general}
%%%%%%%%%%%%%%%%%%%%%%%%%%%%%%%%%

In the previous section we characterized well-covered direct products of connected graphs when neither factor has an
isolatable vertex.  Note that if connected graphs $G$ and $H$ both have girth at least 4 and also have no isolatable vertices, then
it follows from Corollary~\ref{cor:bothcomplete} that $G=H=K_2$.  In the main result of this section we make no assumptions about
the girth of the factors  and no assumptions about whether the factors have isolatable vertices.  We show that if $G$ and $H$
are nontrivial connected graphs whose direct product is well-covered but not very well-covered, then both $G$ and $H$ have girth 3.

\begin{lem} \label{lem:nobipartite}
Let $G$ and $H$ be nontrivial, connected graphs such that $G \DP H$ is well-covered but not very well-covered.  If
$I$ is any independent set of $G$ and $B$ is any bipartite component of $G-N[I]$, then $B=K_1$.
\end{lem}

\begin{proof}
Let $G$ and $H$ be nontrivial, connected graphs such that $G \DP H$ is well-covered but not very well-covered. By Theorems~\ref{thm:wcdirect}
and~\ref{thm:verywcdirect}, $G$ and $H$ are well-covered but neither $G$ nor $H$ is very well-covered.
Suppose the lemma is not true.
Let $I$ be independent in $G$ and let $B$ be a nontrivial bipartite component of $G-N[I]$.  By choosing a maximal independent
set in each of the other components of $G-N[I]$ besides $B$  (if there are any), and adding them to $I$ we get an independent set $J$
such that $B=G-N[J]$.
Since $J \times V(H)$ is independent and $H$ has no isolated vertices,
\[G \DP H -N[J \times V(H)]=(G - N[J]) \DP H = B \DP H\,.\]
It follows from Lemma~\ref{lem:basic} and Theorem~\ref{thm:wcdirect} that  $B \DP H$ and $B$ are well-covered.  Since $B$ is bipartite, we
infer that $B$ is very well-covered.  As a result, $H$ is very well-covered by Theorem~\ref{thm:verywcdirect}, which is a contradiction.
Therefore, $B=K_1$.
\end{proof}

\begin{thm} \label{thm:everyedge}
Let $G$ and $H$ be nontrivial, connected graphs such that $G \DP H$ is well-covered but not very well-covered. Every
edge of $G$ is incident with a triangle.
\end{thm}

\begin{proof}
Let $G$ and $H$ be graphs satisfying the hypotheses of the theorem.  Let $xy$ be an arbitrary edge of $G$.  Note
that $G$ and $H$ are both well-covered but neither is very well-covered.
If $xy$ is in a triangle of $G$, then there is nothing to prove.  Suppose then that $N(x) \cap N(y)=\emptyset$.
Let $I$ be any maximal independent set of the graph $G-N[\{x,y\}]$ and let $F=G-N[I]$.  Note that $F$ is connected.  By Lemmas~\ref{lem:basic}
and~\ref{lem:nobipartite}, $F$ is well-covered and $F$ is not bipartite.  Furthermore, if $y$ is a leaf of $F$, then $\{x\}$ is a maximal independent
set of $F$, but $\{y\}$ is independent but not maximal independent in $F$, which contradicts the fact that $F$ is well-covered.
Thus, $y$ (and similarly $x$) is not a leaf in $F$.   Let $F_x=F-N[y]$ and let $F_y=F-N[x]$.  Since $F$ is not bipartite,
at least one of $F_x$ or $F_y$ contains an edge.  Consequently, at least one of $x$ or $y$ is in a triangle.
\end{proof}

The following corollary follows immediately from Theorem~\ref{thm:everyedge}.

\begin{cor} \label{cor:girth4}
Let $G$ and $H$ be nontrivial, connected graphs. If $G\times H$ is well-covered but not very well-covered, then both $G$
and $H$ have girth $3$.
\end{cor}

From the above result, classifying all well-covered direct products where both factors contain isolatable vertices reduces to the study of
well-covered direct products where both factors have girth $3$ and every edge of $G$ (and of $H$) is incident with a triangle. We now show the existence
of graphs $G$ and $H$ each with girth $3$ and containing isolatable vertices neither of which are very well-covered such that $G\times H$ is well-covered.
The following lemma will be used in the subsequent result.  Its proof is immediate.
\begin{lem} \label{lem:twins}
Let $I$ be a maximal independent set in a graph $G$.  If $u$ and $v$ are two vertices with $N_G(u)=N_G(v)$, then either
$I \cap N_G(u) \not = \emptyset$ or $\{u,v\} \subseteq I$.
\end{lem}

As observed in~\cite{TV-1992}, for $n \ge 2$ the graph $K_n \times K_n$ is well-covered. That is, the ``direct product square'' of a nontrivial
complete graph is well-covered.  The following proposition gives another infinite class of graphs whose direct product squares are well-covered but not
very well-covered.  Of course, any such graph (other than a complete graph) must have isolatable vertices by Theorem~\ref{thm:noisolatable}.

\begin{prop}  \label{prop:multipartite}
Let $r$ and $m$ be  positive integers and let $G$ be the complete $m$-partite graph $K_{r,\ldots,r}$.  The
direct product $G \times G$ is well-covered.
\end{prop}
\begin{proof}
We prove the statement of the proposition for $m=3$ and any $r$.  The proof for an arbitrary $m$ is similar.  Let $V(G)=\{a_1,\ldots,a_{3r}\}$, with
color classes $X_1=\{a_1,\ldots,a_r\}$, $X_2=\{a_{r+1},\ldots,a_{2r}\}$ and $X_3=\{a_{2r+1},\ldots,a_{3r}\}$.  Suppose that $I$ is any maximal
independent set of $G \times G$.  We assume without loss of generality that $(a_1,a_1) \in I$.  The open neighborhood of $(a_1,a_1)$ is
$(X_2 \cup X_3) \times (X_2 \cup X_3)$, and this is the open neighborhood of every vertex in $X_1 \times X_1$.  By Lemma~\ref{lem:twins} it
follows that $X_1 \times X_1 \subseteq I$.  Since $I$ is a maximal independent set, $I$ has a nonempty intersection with exactly one of
$X_1 \times X_2$ or $X_2 \times X_1$.  Again with no loss of generality we may assume that $I \cap (X_1 \times X_2) \not = \emptyset$.  Using
Lemma~\ref{lem:twins} again we can infer that $I=X_1 \times V(G)$.  That is, $|I|=3r^2$, and thus $G\times G$ is well-covered.
\end{proof}


\medskip


\begin{thebibliography}{99999}

\bibitem{b-1981}
    Berge, Claude,
    Some common properties for regularizable graphs, edge-critical graphs and B-graphs.
    Graph theory and algorithms (Proc. Sympos., Res. Inst. Electr. Comm.,  Tohoku Univ., Sendai, 1980),
    Lecture Notes in Comput. Sci., Springer, Berlin-New York 108, 108--123  (1981).

\bibitem{CER-1993}
   Campbell, S.~R., Ellingham, M.~N.,  Royle, G.~F.:
   A characterisation of well-covered cubic graphs.
   J. Combin. Math. Combin. Comput. 13, 193--212 (1993)

\bibitem{f-1982}
   Favaron, O.:
   Very well covered graphs.
   Discrete Math.  42  no. 2--3, 177--187 (1982)

\bibitem{FHN-1993}
    Finbow, A., Hartnell, B.~L., Nowakowski, R.~J.:
    A characterization of well-covered graphs of girth {$5$} or greater.
    J. Combin. Theory Ser. B. 57, 44--68 (1993)

\bibitem{FHN-1994}
    Finbow, A., Hartnell, B.~L., Nowakowski, R.~J.:
    A characterization of well-covered graphs that contain neither {$4$}- nor {$5$}-cycles.
    J. Graph Theory. 18, 713--721 (1994)

\bibitem{FHNP-1}
    Finbow, A., Hartnell, B.~L., Nowakowski, R.~J., Plummer, M.~D.:
    On well-covered triangulations. I. Stability in graphs and related topics.
    Discrete Appl. Math.  132  no. 1-3, 97--108  (2003)

\bibitem{FHNP-2}
    Finbow, A., Hartnell, B.~L., Nowakowski, R.~J., Plummer, M.~D.:
    On well-covered triangulations. II.
    Discrete Appl. Math.  157  no. 13, 2799--2817 (2009)

\bibitem{FHNP-3}
    Finbow, A., Hartnell, B.~L., Nowakowski, R.~J., Plummer, M.~D.:
    On well-covered triangulations. III.
    Discrete Appl. Math.  158  no. 8, 894--912 (2010)
		
\bibitem{FHNP-4}
   Finbow, A., Hartnell, B.~L., Nowakowski, R.~J., Plummer, M.~D.:
   Well-covered triangulations: Part IV.
   Discrete Appl. Math.  215,   71--94 (2016)

\bibitem{F-2009}
    Fradkin, A.~O.:
    On the well-coveredness of {C}artesian products of graphs.
    Discrete Math. 309, 238--246 (2009)

\bibitem{HR-2013}
    Hartnell, B., Rall, D.~F.:
    On the {C}artesian product of non well-covered graphs.
    Electron. J. Combin. 20, Paper 21, 4 pages (2013)

\bibitem{HRW-2018+}
    Hartnell, B., Rall, D.~F., Wash, K.:
    On well-covered {C}artesian products.
    Graphs Combin. 34, 1259--1268 (2018)

\bibitem{JS-1994}
   Jha, P.~K., Slutzki, G.:
   Independence numbers of product graphs.
   Appl. Math. Lett. 7(4), 91--94 (1994)

\bibitem{NR-1996}
   Nowakowski, R.~J., Rall, D.~F.:
   Associative graph products and their independence, domination and coloring numbers.
   Discuss. Math. Graph Theory 16, 53--79 (1996)

\bibitem{P-1970}
   Plummer, M.~D.:
   Some covering concepts in graphs.
   J. Combinatorial Theory.  8, 91--98 (1970)

\bibitem{TV-1992}
    Topp, J., Volkmann, L.:
    On the well-coveredness of products of graphs.
    Ars Comb. 33, 199--215 (1992)

\bibitem{W-1996}  West, Douglas B.:  Introduction to Graph Theory.
Prentice Hall, Inc., Upper Saddle River, NJ,  (1996)


\end{thebibliography}
\end{document}